\documentclass[12pt,leqno]{amsart}
\usepackage{a4wide}
\usepackage{amsmath}
\usepackage[italian, english]{babel}
\usepackage{amsfonts}
\usepackage{amssymb}
\usepackage{dsfont}
\usepackage{graphicx}
\usepackage{fancyhdr}
\usepackage{amsthm}
\usepackage{empheq}
\usepackage{cases}
\usepackage[all]{xy}
\usepackage{stmaryrd}
\usepackage[colorlinks, citecolor=black, linkcolor=black]{hyperref}
\usepackage{mathrsfs}

\usepackage{mathtools}
\usepackage{tikz-cd}

\usepackage{eucal}

\usepackage{amsaddr}

\DeclareMathAlphabet\mathbfcal{OMS}{cmsy}{b}{n}

\newtheorem{definition}{Definition}[section]
\newtheorem{theorem}{\textbf{Theorem}}[section]
\newtheorem{lemma}[theorem]{\textbf{Lemma}}
\newtheorem{proposition}[theorem]{\textbf{Proposition}}
\newtheorem{corollary}[theorem]{\textbf{Corollary}}
\newtheorem{remark}[theorem]{\textbf{Remark}}
\theoremstyle{remark}

\numberwithin{equation}{section}

\title{Spin frame transformations and Dirac equations} 

\author{R.Noris$^{(1)(2)}$, L.Fatibene$^{(2)(3)}$}
\address{$^{(1)}$ DISAT, Politecnico di Torino,
C.so Duca degli Abruzzi 24, I-10129 Torino, Italy \\
$^{(2)}$INFN Sezione di Torino, 
Via Pietro Giuria 1, I-10125 Torino, Italy\\
$^{(3)}$ Dipartimento di Matematica -- University of Torino,\\
via Carlo Alberto 10, I-10123 Torino, Italy
}
\email{ruggero.noris at polito.it, lorenzo.fatibene at unito.it}

\begin{document}

\begin{abstract}
We define spin frames, with the aim of extending spin structures from the category of (pseudo-)Riemannian manifolds to the category of spin manifolds with a fixed signature on them, though with no selected metric structure. Because of this softer requirements, transformations allowed by spin frames are more general than usual spin transformations and they usually do not preserve the induced metric structures.\\
We study how these new transformations affect connections both on the spin bundle and on the frame bundle and how this reflects on the Dirac equations.
\end{abstract}

\maketitle

\section{Introduction}

Dirac equations provide an important tool to study the geometric structure of manifolds, as well as to model the behaviour of a class of physical particles, namely fermions, which includes electrons.\\
The aim of this paper is to generalise a key item needed to formulate Dirac equations, the spin structures, in order to extend the range of allowed transformations. Let us start by first reviewing the usual approach to Dirac equations.\\
Let $(M,g)$ be an orientable pseudo-Riemannian manifold with signature $\eta=(r,s)$, such that $r+s=m=dim(M)$. \\
Let us denote by $L(M)$ the (general) frame bundle of $M$, which is a $GL(m,\mathbb R)$-principal fibre bundle. Points of $L(M)$ are pairs $(x,\epsilon_a)$, where $\epsilon_a$ is a base of $T_xM$. Furthermore, if $M$ is orientable or, more generally, if we are able to single out a class of positive frames, one can define the subbundle $SO_e(M,g)\subset L(M)$ of orthonormal frames with respect to $g$. It is a principal $SO_e(r, s)$-bundle, where $SO_e(r, s)$ denotes the connected component to the identity of the orthogonal group $ O(r, s)$ in the relevant signature which, depending on the case, may have more than one connected component. 
One has a canonical embedding $\hat \imath:SO_e(M, g)\rightarrow L(M)$, which is a reduction of the frame bundle structure group \cite{KobaNu} satisfying
\begin{center}
\begin{tikzcd}
SO_e(M,g) \arrow{r}{\hat \imath} \arrow{d}[left]{\hat\pi} & L(M)\arrow{d}[left]{\pi}\\
M \arrow[r,equal] & M
\end{tikzcd}
\quad
\begin{tikzcd}
SO_e(M,g) \arrow[r,"\hat \imath"] \arrow[d,"R_g"] & L(M)\arrow[d,"R_{i(g)}"]\\
SO_e(M,g) \arrow[r,"\hat \imath"] & L(M)
\end{tikzcd}
\end{center} 
where $i:SO_e(r,s)\hookrightarrow GL(m,\mathbb R)$ is the canonical group embedding.\\
Once a certain signature is fixed, one has a canonical matrix $\eta_{ab}=\rm diag(-1, \dots, -1, 1, \dots,  1)$ with $r$ pluses and $s$ minuses, which in turn defines a bilinear form $\eta = \eta_{ab} E^a\otimes E^b$  in any vector space $W$ of dimension $m$ with a basis $E_a$, which is then $\eta$-orthonormal. One can then define the relevant Clifford algebra $Cl(r, s)$ and the corresponding $Spin(r, s)$ group in it. This allows to build the spin bundle, which we will denote with $P$, a principal fiber bundle $(P,M,p,Spin_e(r,s))$ which is a main ingredient in the definition of spin structures.
\begin{definition}
A spin structure over $(M,g)$ is a pair $(P,\Lambda)$, where P is a spin bundle and $\Lambda :P\to SO_e(M,g)$ is a vertical principal morphism with respect to the covering map $l: Spin_e(r,s)\to SO_e(r,s)$. 
\end{definition} 
This is equivalent to require the commutativity of the following diagrams
\begin{center}
\begin{tikzcd}
P \arrow{r}{\Lambda} \arrow{d}{Spin_e(r,s)}[left]{p} & SO_e(M,g)\arrow{d}{SO_e(r,s)}[left]{\hat{\pi}}\\
M \arrow[r,equal] & M
\end{tikzcd}
\quad
\begin{tikzcd}
P \arrow[r,"\Lambda"] \arrow[d,"R_g"] & SO_e(M,g)\arrow[d,"R_{l(g)}"]\\
P \arrow[r,"\Lambda"] & SO_e(M,g)
\end{tikzcd}
\end{center} 
\begin{remark}
A pseudo-Riemannian manifold $(M,g)$ admits a spin structure $(P,\Lambda)$ if and only if its second Stiefel-Whitney class is zero: this condition is independent of the metric and of the signature. A manifold satisfying it is called spin manifold. In the literature, spin structures are used to formulate Dirac equations on manifolds: in order to use these structures, one has then to restrict the choice of base manifolds to those satisfying the topological condition $w_2(M)=0$. Further details can be found in \cite{LawsonMichelsohn} and \cite{Antonsen}.
\end{remark}
Let us then select a spin structure $(P, \Lambda)$ and consider a representation $\hat\lambda: Cl(r, s)\times V\rightarrow V$ of the Clifford algebra on a (usually complex) vector space of dimension $k=dim(V)$, 
which restricts to a representation $ \lambda: Spin_e(r, s)\times V\rightarrow V$ of the (connected component to the identity of the) spin group. One can then define an associated vector bundle, the {\it spinor bundle}, whose sections will be accordingly called {\it spinor fields}.
\begin{definition}
Let P be a spin bundle and $\lambda$ a representation of the spin group on a vector space V. The spinor bundle is the vector bundle associated to the spin bundle $P$ 
\begin{align}
S(P)=P\times_\lambda V.
\end{align}
\end{definition}
Notice that, once a representation $\hat \lambda$ is given, one can also associate matrices to the Clifford generators ${\bf e}_a$, which will be denoted by $\gamma^a:= \eta^{ab} \hat \lambda({\bf e}_b)$ and are (generally complex) $k\times k$ matrices.
At last, given a spin structure, a principal connection $\hat H_p\in T_pP$ and a representation $\hat\lambda$ one can define the Dirac equation as 
\begin{align}\label{diraceq}
i e_a^\mu \gamma^a  \hat\nabla_\mu \psi +\mu\psi =0,
\end{align}
 where $(x,\partial_\mu)e^\mu_a:=\Lambda(\overset{\scriptscriptstyle(\alpha)}{\sigma})$, $\overset{\scriptscriptstyle(\alpha)}{\sigma}$ being a local section of $P$ and $\hat\nabla_\mu \psi$ denotes the covariant derivative of the spinor field with respect to the induced connection on the associated bundle $S(P)$.\footnote{As noted from the previous formula, we are following Einstein's summation convention for contracted indices, which means, for example, that repeated indices are summed over
\begin{align}\nonumber
(x,\partial_\mu)\overset{\scriptscriptstyle(\alpha)}{e} {}_a^\mu \longrightarrow \sum_{\mu=1}^m(x,\partial_\mu)\overset{\scriptscriptstyle(\alpha)}{e} {}_a^\mu.
\end{align}
The range of both Latin and Greek indices will be $\alpha, a=1,...,m$ if not otherwise specified. }\\
Equation \eqref{diraceq} is clearly defined locally, i.e.~with respect to a local trivialization of $P$, but one can easily prove that it is covariant with respect to automorphisms of $P$ (active viewpoint) or, equivalently, with respect to changes of local trivializations on $P$ (passive viewpoint). This is briefly proved in the Appendix. \\
In view of the fact that automorphisms of $P$ are symmetries for the Dirac equation, we call them {\it spin transformations}. Under these transformations, because of equivariance of the spin structures, the frame $e_a$, used to define Dirac equation, is changed to a different $g$-orthonormal frame or, equivalently, one could say that the metric is left unchanged under spin transformations.\\
Let us also remark that in this way, Dirac theory is cast as a gauge-natural field theory; see for example \cite{Book}, \cite{Kolar}. Roughly speaking, that means that one has a configuration bundle which is a gauge-natural bundle associated to $P$ and the dynamics is covariant with respect to all automorphisms of $P$. \\
The usual approach to Dirac theory arises, as we have seen, through the use of spin structures, which, by definition, require a metric to be defined in the first place. \\

\par There is also a different, more general and somehow more geometric framework for Dirac equations, based on the universal covering $\tilde L(M)$ of the frame bundle $L(M)$. It is interesting to briefly review it and also to compare to it; see \cite{muller}, \cite{trautman} and \cite{percacci}.\\
The universal covering of the (positive) frame bundle, denoted by $\tilde L(M)$, is a two-fold covering and it is a principal bundle for the group $\tilde{GL}(m)$, which is a double covering of $GL(m)$. \\
{
As we have said, fixing a metric $g$ (of signature $(r,s)$) on a manifold $M$ defines the sub-bundle $SO(M, g)$ of (positive) $g$-orthonormal frames. While $SO(M, g)\subset L(M)$ is a principal bundle for the group $SO(r, s)$, its preimage $\tilde \Sigma_g \subset \tilde L(M)$ is a principal bundle for the group $Spin(r, s)$
so that $(\tilde \Sigma_g, \tilde\pi\circ \tilde\imath|_{SO(M,g)})$ is a standard spin structure, where $\tilde\pi\circ \tilde\imath|_{SO(M,g)}: \tilde \Sigma_g\rightarrow SO(M, g)$ is the covering map obtained by the immersion $\tilde\imath: \tilde\Sigma_g\rightarrow \tilde L(M)$ and the covering projection $\tilde\pi: \tilde L(M)\rightarrow L(M)$.}

In this more general framework, one does not need to fix a metric {\it a priori}, but on the contrary it aims to catch all spin structures at once.
In particular, one can lift a diffeomorphism of $M$ to $L(M)$ by using its natural structure, and to $\tilde L(M)$ by using the covering properties.\\
As a result, in this framework one can systematically define isomorphisms of spin structures, possibly with different underlying metrics and define natural-like properties.\\
However, all spin structures end up in the same framework, even those with non-isomorphic $Spin(r,s)$-bundles. Depending on the problem one needs to study, this can be a good feature or a drawback. For example, in a variational setting, it is often better to define configurations with the same global properties, on one single structure bundle. In fact it is for a similar reason that, when discussing electromagnetism, the formalism based on principal $U(1)$-connections is better that a formulation in terms of forms, even though Maxwell equations can be formulated in terms of forms and the latter framework often contains solutions which, when defined as connections, are on different $U(1)$-bundles.\\
While the standard spin structure formulation is suitable for an {\it a priori} fixed metric structure, the double covering formulation is suitable to discuss spin structures on a general geometry, where one can change the metric and the global properties of spinors.\\
In this paper we will define a more general formalism than the standard spin structures, which allows different metrics at once, though it separates solutions which are defined on different structure bundles $P$. As it happens in electromagnetism, by doing that one gives up naturality (i.e.~covariance with respect to diffeomorpfisms on the manifold $M$), though replacing it with covariance with respect to automorphisms of the structure bundle $P$, as in gauge theories. This kind of covariance is called {\it gauge-naturality} and let us trivially remark that the group of general automorphims of $P$ projects on the group of diffeomorphisms of $M$, while an embedding of diffeomorphisms in automorphisms in not naturally defined.\\
Our formulation sits more or less halfway between the standard spin structures and double coverings, maintaining some of the properties (e.g.~allowing different metrics), but not all (e.g.~we work with a fixed structure bundle $P$). It naturally takes into account that one can change reference frames without changing the point on the manifold (which corresponds to a vertical automorphisms) and Dirac equations are clearly covariant with respect to these transformations as well. These act as gauge transformations in electromagnetism. \\
This formulation, the spin frame approach, will be therefore introduced in the next Section, where we will see how one can define a metric on a bare manifold. We want to avoid fixing a metric structure {\it a priori} on $M$ because we have in mind possible applications to physics and in particular to general relativistic theories for which a background independent formulation is preferred. An approach for which a metric $g$ is determined \textit{a posteriori} is indeed needed, when discussing Dirac equation in interaction with gravity: in that case, spin frames (and thus the induced metric) are unknown until the system of field equations (Einstein and Dirac, which are coupled PDE) is solved.\\
In Section \ref{sec3} we will define {\it spin frame transformations}, which extend the group of allowed transformations on spin frames. These contain transformations which, unlike spin transformations, do change the metric.\\
In Section \ref{sec4} we will investigate how spin frame transformations act on connections on $P$ and how they affect the correspondence between them and connections of $L(M)$.\\
Such correspondence can be used to transform Dirac equations, which seems to have applications in describing graphene: quasiparticles can be described as fermions satisfying Dirac equations in non-flat geometries and non-trivial spin structures.
We shall briefly discuss these applications in the last Section.

\tableofcontents

\section{Spin frames}\label{sec2}
As argued in the previous Section, we aim here to rephrase spin structures so that the construction starts from a spin manifold $M$ and a signature $(r, s)$, but with no fixed metric on it.\\
Notice that we are not claiming that spin frames are essentially different from spin structures. On the contrary, they are essentially the same thing: in some contexts, the difference is not essential. In many situations a Riemannian structure is simply there and spin structures are enough to define global Dirac equation to be studied.\\
In other situations \cite{Spinors}, \cite{Jadwisin}, \cite{DoBosons} one needs Dirac equation to be coupled with Einstein equations to form a coupled system of PDE and the metric structure is unknown \textit{a priori}, because it has to be obtained as a solution. In these cases, the difference between spin structures and spin frames is relevant. 
\begin{definition}\label{spinframe}
A spin frame over the spin manifold M is a pair $(P,e)$, where P is a spin bundle and $e: P\to L(M)$ is a vertical principal morphism with respect to the map $i\circ l:Spin_e(r,s)\rightarrow GL(m,\mathbb R)$.
\end{definition}
In other words, we must require the commutativity of the following diagrams
\begin{center}
\begin{tikzcd}
P \arrow{r}{e} \arrow{d}{Spin_e(r,s)}[left]{p} & L(M)\arrow{d}{GL(m,\mathbb{R})}[left]{\pi}\\
M \arrow[r,equal] & M
\end{tikzcd}
\quad
\begin{tikzcd}
P \arrow[r,"e"] \arrow[d,"R_g"] & L(M)\arrow[d,"R_{i\circ l(g)}"]\\
P \arrow[r,"e"] & L(M)
\end{tikzcd}
\end{center} 
As one can see from the definition, no metric structure is required on the base manifold $M$. Let us now fix the notation and prove some properties of spin frames. Let $U_\alpha$ be an open cover of $M$ and let $\overset{\scriptscriptstyle(\alpha)}{\sigma} :U_\alpha\subset M\to P$ be a family of local sections of $P$, defined on each open of $M$. Being $P$ principal, on the overlap of two patches $U_\alpha\cap U_\beta$ one must have $$\overset{\scriptscriptstyle(\beta)}{\sigma}=\overset{\scriptscriptstyle(\alpha)}{\sigma}\cdot\varphi_{(\alpha\beta)},$$ where $\varphi_{(\alpha\beta)}$ are the transition functions on $P$ valued on $Spin_e(r,s)$.\\
The image of a spin frame is completely determined once evaluated on such a family of local sections, as it can be extended by equivariance. Furthermore, a spin frame always induces local sections on the frame bundle
\begin{align}\label{inducedsection}
e(\overset{\scriptscriptstyle(\alpha)}{\sigma})=:(x,\overset{\scriptscriptstyle(\alpha)}{e}_a),
\end{align}
where $a=1,...,m$. These induced sections can obviously be expressed in terms of the natural section $(x,\partial_\mu)$, induced by the local coordinates $x^\mu$ on $U_\alpha\subset M$, where $\mu=1,...,m$ as 
\begin{align}\nonumber
(x,\overset{\scriptscriptstyle(\alpha)}{e}_a)=(x,\partial_\mu)\overset{\scriptscriptstyle(\alpha)}{e} {}_a^\mu.
\end{align}
Consider now the left group action 
\begin{align}
\rho :&\ (Spin(\eta)\times GL(m,\mathbb R))\times GL(m,\mathbb R)\to GL(m,\mathbb R)\\ \nonumber
:&\ (S,J,e)\mapsto e'= J\cdot e\cdot \hat l(S)=:\rho(S,J,e),
\end{align}
where $\hat l=i\circ l: Spin_e(r,s)\to GL(m,\mathbb R)$. This allows to look at spin frames as sections of an associated vector bundle.
\begin{definition}
The bundle of spin frames is the associated bundle to the principal bundle $(P\times_M L(M),M,p,Spin_e(r,s)\times GL(m,\mathbb R))$ through the group action $\rho$. We will denote it as
\begin{align} \nonumber
F(P)=(P\times_M L(M))\times_\rho GL(m,\mathbb R)
\end{align}
\end{definition}
\begin{remark}
Before continuing with the following proposition, let us stress out that a certain mathematical object is global if it is expressed without any reference to local coordinates or trivializations or if its local expressions satisfy certain glueing properties in the overlap of different patches or trivializations. As an example, take a global vector field  $X$ defined on a manifold $N$: then it is global if in any intersection of local patches $V_\alpha\cap V_\beta$ the components of the vector field transform as $$X'^\mu=J^\mu_\nu X^\nu,$$ where $J^\mu_\nu$ is the Jacobian of the coordinates transformation between the two local patches. \\
In this paper we will try and keep local expressions confined in remarks. Nevertheless, such local expressions can be useful, despite hiding the global nature of objects, which is always ensured by the transformation laws obeyed by the objects.
\end{remark}
\begin{proposition}\label{propF(P)}
There is a one-to-one correspondence between spin frames $(P,e)$ and sections of $F(P)$.
\end{proposition}
\begin{proof}
If a given spin frame $(P,e)$ is a global morphism then its local expressions must appropriately glue on the overlap of different patches of $M$. The glueing condition is obtained by considering equivariance of spin frames
\begin{align} \nonumber
e(\overset{\scriptscriptstyle(\beta)}{\sigma})=(x,\overset{\scriptscriptstyle(\beta)}{e}_a)=(x,\overset{\scriptscriptstyle(\alpha)}{e}_b\hat l\ ^b_a(\varphi_{(\alpha\beta)}))\implies (x,\partial_\mu ')\overset{\scriptscriptstyle(\beta)}{e}{}^\mu_a=(x,\partial '_\mu) J^\mu_\nu \overset{\scriptscriptstyle(\alpha)}{e}{}^\nu_b\hat l\ ^b_a(\varphi_{(\alpha\beta)}),
\end{align}
which yields
\begin{align}\label{changetrivialization}
\overset{\scriptscriptstyle(\beta)}{e} {}^\mu_a=J^\mu_\nu \overset{\scriptscriptstyle(\alpha)}{e}{}^\nu_b\hat l\ ^b_a(\varphi_{(\alpha\beta)}).
\end{align}
On the other hand, a section of $F(P)$ is an equivalence class 
\begin{align}
[((\varphi,J),e)]_\rho=[((Id_{Spin(\eta)},Id_{GL(m,\mathbb R)}),e'=\rho(\varphi,J,e))]_\rho.
\end{align}
Since the glueing condition and the representative of the equivalence class are in the same form, we see that, given a global spin frame, a section of $F(P)$ is defined and vice versa we see that, given an equivalence class in $F(P)$, which is described by all $e'=J\cdot e\cdot \hat l(S)$ for any $(\varphi,J)\in Spin_e(r,s)\times GL(m,\mathbb R)$, a global spin frame is defined.
\end{proof}
Let us now see how spin frames are able to define a metric structure: this will be uniquely defined, whereas the opposite is not possible. Indeed, given a metric, one can only find an orthonormal frame, up to orthogonal transformations.
\begin{definition}
Let $(P,e)$ be a spin frame and $p\in P$ a point in the spin bundle mapped into $e(p)=(x,\epsilon_a)\in e(P)\hookrightarrow L(M)$. We define the induced metric as
\begin{align}\label{metricdefinition}
g(\epsilon_a,\epsilon_b):=\eta_{ab}.
\end{align}
\end{definition}
\begin{remark}
The intrinsic definition \eqref{metricdefinition} can be locally restated as
\begin{align}\label{localg}
g=g_{\mu\nu}dx^\mu\otimes dx^\nu, \qquad g_{\mu\nu}:=\epsilon^a_\mu\eta_{ab}\epsilon^b_\nu,
\end{align}
where $x^\mu$ are the coordinates in a patch $U_\alpha\subset M$.
\end{remark}
We refer to the Appendix for a proof that spin transformations send $g$-orthonormal frames into $g$-orthonormal frames, a property which of course also holds for spin structures.\\
We finally focus on existence conditions of spin frames by proving the following theorem.
\begin{theorem}
A spin frame $(P,e)$ over the manifold $M$ exists if and only if there exists a spin structure over $(M,g)$ for some metric $g$ on $M$.
\end{theorem}
\begin{proof}
Given a spin frame $(P,e)$ over $M$, one can consider the image $e(P)=\{(x,e_a): \ e_a \ \textit{\rm is a basis of}\  T_xM\ \wedge \ g(e_a,e_b)=\eta_{ab} \}$, where we crucially remark that the metric appearing here is the one induced by the spin frame. One can then define a spin structure as
\begin{center}
\begin{tikzcd}
P \arrow{r}{\hat e} \arrow{d}[left]{p} & e(P)\arrow{d}[left]{\pi|_{e(P)}}\\
M \arrow[r,equal] & M
\end{tikzcd}
\quad
\begin{tikzcd}
P \arrow[r,"\hat e"] \arrow[d,"R_g"] & e(P)\arrow[d,"R_{i\circ l(g)}|_{SO_e(r,s)}=R_{l(g)}"]\\
P \arrow[r,"\hat e"] & e(P)
\end{tikzcd}
\end{center} 
On the contrary, given a spin structure $(P,\Lambda)$ over $(M,g)$ one can define a spin frame $(P,e=\hat \imath\circ\Lambda)$: this is indeed a spin frame since
\begin{align}\nonumber
e\circ R_g=\hat \imath\circ\Lambda \circ R_g=\hat \imath\circ R_{l(g)}\circ\Lambda=R_{i\circ l(g)}\circ\hat \imath\circ\Lambda=R_{i\circ l(g)}\circ e.
\end{align}
\end{proof}
This is the same as saying that the following diagrams commute
\begin{center}
\begin{tikzcd}
P \arrow[r,"\Lambda"] \arrow{d}[left]{p} & SO_e(M,g)\arrow{d}[left]{\hat\pi} \arrow[r,"\hat\imath"] &L(M)\arrow{d}[left]{\pi} \\
M \arrow[r,equal] & M\arrow[r,equal] &M
\end{tikzcd}
\quad
\begin{tikzcd}
P \arrow[r,"\Lambda"] \arrow[d,"R_g"] & SO_e(M,g)\arrow[d,"R_{l(g)}"]\arrow[r,hook,"\hat \imath"]& L(M) \arrow[d,"R_{i\circ l(g)}"]\\
P \arrow[r,"\Lambda"] & SO_e(M,g)\arrow[r,hook,"\hat \imath"]& L(M) 
\end{tikzcd}
\end{center} 
Since existence conditions between spin structures and spin frames are equivalent, one might still ask why bother considering the latter. We now reinforce the motivations we have given before by considering a one parameter family of spin structures $(P,\Lambda_t)$, where $t\in\mathbb R$. These are maps from the spin bundle to a given orthonormal frame bundle, each one differing by an orthogonal transformation. \\
A one parameter family of spin frames is instead a couple $(P,e_t)$: these are maps from the spin bundle into the frame bundle. By restricting these maps to their images, which are different orthonormal frame bundles $e_t(P)$, one obtains a family of spin structures, as shown in the previous theorem. We remark that the obtained spin structures differ one another for the arising metric $g_t$, which was instead the same in the case of a one parameter family of spin structures. In the following Section we will introduce transformations connecting two different orthonormal frame bundles. 


\section{Spin frame transformations and spin bundle connections}\label{sec3}
We now define spin frame transformations and prove that, contrary to spin transformations, they are able to induce different metric structures on $M$.
\begin{definition}
Let $(P,e)$ be a spin frame and $\Phi:L(M)\to L(M)$ a vertical principal automorphism of the frame bundle. We define $(P,\tilde e)$ as the transformed spin frame that makes the following diagram commute.
\begin{center}
\begin{tikzcd}
                                                      &  &                            & L(M) \arrow[dd,"\pi"] \\
P \arrow[rr, "e"] \arrow[rrru, "\tilde e"] \arrow[dd,"p"] &  & L(M) \arrow{ru}[near start]{\Phi} \arrow[crossing over]{dd}{\pi} &                 \\
                                                      &  &                            & M               \\
M \arrow[rr,equal]                                        &  & M \arrow[ru,equal]               &                
\end{tikzcd}
\end{center} 
We will refer to the automorphism $\Phi$ as a spin frame transformation.
\end{definition}  
\begin{remark}
Being a vertical automorphism, we know that its action is completely determined if we know it on local sections of the frame bundle. \\
Take for example the one induced by the spin frame $(P,e)$, $(x,\overset{\scriptscriptstyle(\alpha)}{e}{}_a)$: then
\begin{align}\nonumber
(x,\overset{\scriptscriptstyle(\alpha)}{\tilde e}{}_a)=\tilde e(\overset{\scriptscriptstyle(\alpha)}{\sigma})&=\Phi(e(\overset{\scriptscriptstyle(\alpha)}{\sigma}))=\Phi(x,\overset{\scriptscriptstyle(\alpha)}{e}{}_a)=(x,\overset{\scriptscriptstyle(\alpha)}{e}{}_b)\phi^b_a(x),
\end{align}
where $\phi^b_a$ is a $GL(m,\mathbb R)-$matrix. A generic point $e(p)=(x,\epsilon_a)=(x,\overset{\scriptscriptstyle(\alpha)}{e}{}_b)\epsilon^b_a$ is then mapped to
\begin{align}\nonumber 
(x,\overset{\scriptscriptstyle(\alpha)}{e}{}_c)\tilde\epsilon^c_a=\tilde e(p)=\Phi(e(p))=\Phi((x,\overset{\scriptscriptstyle(\alpha)}{e}{}_b)\epsilon^b_a)=\Phi(x,\overset{\scriptscriptstyle(\alpha)}{e}{}_b)\epsilon^b_a=(x,\overset{\scriptscriptstyle(\alpha)}{e}{}_c)\phi^c_b(x)\epsilon^b_a,
\end{align}
which is locally expressed by the left action 
\begin{align}
\Phi: \begin{cases}
      x'=x\\
       \tilde\epsilon^c_a=\phi^c_b(x)\epsilon^b_a.
    \end{cases}       
\end{align}
The same procedure can be repeated for the natural section $(x,\partial_\mu)$, which yields 
\begin{align}
\Phi: \begin{cases}
      x'=x\\
      \tilde\epsilon^\mu_a=\phi^\mu_\nu(x) \epsilon^\nu_a,
    \end{cases}       
\end{align}
where again $\phi^\mu_\nu$ is a $GL(m,\mathbb R)-$matrix.
\end{remark}
We now prove that spin frames induce different metric structures. Metrics on a manifold can be thought of as points of the associated bundle $Lor(M)=L(M)\times_\rho L$, where $L\subset T^2_0(\mathbb R^m)$ is the set of symmetric, non-degenerate tensors with signature $(r,s)$ and 
\begin{align}\nonumber
\rho: \ &GL(m,\mathbb R)\times L\to L\\ \nonumber
&(A^a_b,s)\mapsto \rho(A^a_b,s):=s(\bar A^a_c v_a,\bar A^b_d v_b)w^c\otimes w^d,
\end{align}
where $v_a$ and $w^a$ are a basis of vectors and linear forms on $\mathbb R^m$, respectively. Without loss of generality, we can choose the basis of the linear forms to be $dx^\mu$, induced by the coordinates on $\mathbb R^m\simeq T_x^*M$. Furthermore, once a symmetric, non-degenerate tensor is given, one can freely choose a global basis of vectors $t_a\in\mathbb R^m\simeq T_xM$ such that $s(t_a,t_b)=\eta_{ab}$.\\
From the theory of associated bundles we know that principal morphisms induce morphisms on the associated bundles. From this fact we now prove the following proposition.
\begin{proposition}
A generic spin frame transformation acting on a spin frame $(P,e)$ induces a different metric structure on the base manifold.
\end{proposition}
\begin{proof}
Let $(x,\epsilon_a)\in L(M)$ and $u=[(x,\epsilon_a),s]_\rho$ be a point of $Lor(M)$. We define the morphism $\hat\Phi$ as
\begin{align}\nonumber
 \hat\Phi: \ Lor(M)&\to Lor(M)\\ \nonumber
u&\mapsto \hat\Phi(u):=[\Phi(x,\epsilon_a),s]_\rho
\end{align}
To show our thesis, we take a trivialization $\hat t_{\partial_\mu}$ induced by the trivialization $t_{\partial_\mu}$ of $L(M)$. Trivialization on the associated bundles are defined as follows
\begin{align}\nonumber
\hat t_{\partial_\mu}(u)&:=(x,\rho(\epsilon^\mu_a,s))=(x,(\epsilon^a_\mu\eta_{ab}\epsilon^b_\nu) dx^\mu\otimes dx^\nu),\\ \nonumber
\hat t_{\partial_\mu}(\hat\Phi(u))&:=(x,\rho(\tilde \epsilon^\mu_a,s))=(x,(\tilde \epsilon^a_\mu\eta_{ab}\tilde \epsilon^b_\nu) dx^\mu\otimes dx^\nu).
\end{align}
If now $(x,\epsilon_a)\in e(P)$ then
\begin{align}\nonumber
\hat t_{\partial_\mu}(u)&=(x,g),\\ \nonumber
\hat t_{\partial_\mu}(\hat\Phi(u))&=(x,\tilde g)=(x,(\bar\phi^\rho_\mu(x) g_{\rho\sigma}\bar\phi^\sigma_\nu(x))dx^\mu\otimes dx^\nu),
\end{align}
where $g$ is the induced metric by the spin frame $(P,e)$. One clearly sees that the two local expressions of the metric are different, being $\Phi$ a principal morphism on $L(M)$. One can prove that this holds for any other chosen trivialization on the frame bundle (take for example the one induced by the spin frame), which implies that the two metrics are indeed different.
\end{proof}
Now that we have established the main properties of spin frame transformations on spin frames and metric structures, consider connections, which are an essential ingredient to build the Dirac equation, as we saw. In many situations the connection can be induced by the metric (or the spin frame), although sometimes it is useful to consider it independent of the metric (Palatini formulation of gravity (see \cite{Palatini} and \cite{Palatini2}), Palatini $f(\mathcal R)$-theories (see \cite{Extended2} and \cite{EG}), Ehlers-Pirani-Schild framework (\cite{EPS} and \cite{DiMauro}), teleparallel gravity \cite{Held}). Furthermore, if we are going to consider spinor fields in interaction with the gravitational field (which is identified with a geometry on $M$) we need to provide some result relating $P$ and $M$. Luckily, we already have a functor to map a spin frame $(P, e)$ to its associated metric structure $(M, g)$.\\
We shall discuss in this Section the case in which the connection is independent of the metric: take indeed a principal connection $\hat H$ on $P$ and use the spin frame $(P,e)$ to define horizontal subspaces on $L(M)$ as $H_{e(p)}:=T_pe(\hat H_p)$.
\begin{remark}
For this to make sense, we have to check that the subspace is well-defined, since we have two (different) points $p,p'\in P$, which are mapped onto the same frame $(x, \epsilon_a)= e(p)= e(p')$. They are related by $p'= p\cdot (-\mathbb I)$, since $ker(l)=ker(\hat l)=\{\pm \mathbb I\}$. \\
We then have $T_pR_{-\mathbb I}(\hat H_p) = \hat H_{p'}$ and 
$$H_{e(p')}=T_{p'}e(\hat H_{p'})=T_{p'}e \circ T_pR_{-\mathbb I} (\hat H_{p})=T_{p} ( e \circ R_{-\mathbb I} ) (\hat H_{p})=T_{p}  e (\hat H_{p}) = H_{e(p)},$$
where we used the map identity $ e \circ R_{-\mathbb I}=R_{\hat l(-\mathbb I)}  \circ e = e : P\rightarrow L(M)$.
\end{remark}
One can easily check that the family of the subspaces  $H_{e(p)}\subset T_{e(p)} L(M)$  can be extended by equivariance to $L(M)$ and that it is, by construction, right invariant. It is known that giving a principal connection on $P$ is equivalent to giving a map $\hat\omega:TM\rightarrow \hat H\subset TP$ which lifts a vector tangent to $M$ to a vector in the given horizontal subspace of $P$. Then 
\begin{align}
\hat H_p=\left\{\hat \omega(v)=v^\mu(\partial_\mu|_p-\omega^{ab}_\mu(x)\sigma_{ab}|_p), \forall v\in T_xM, \pi(p)=x\right\},
\end{align}
where $\sigma_{ab}$ are a set of vertical right invariant vector fields on $P$, thus the index $a=1,...,m$ runs along the Lie algebra $spin(r,s)$ and $\sigma_{(ab)}=0.$\\
We then see that to give a principal connection on $P$ is completely equivalent to give the coefficients $\omega^{ab}_\mu(x)$.\\
We know that a spin frame transformation can define a new spin frame $\tilde e= \Phi \circ e$: this in turn defines another horizontal subspace $T_{p}\tilde e(\hat H_p) $. Clearly 
\begin{align}
\tilde H_{\tilde e(p)} = T_{p}\tilde e(\hat H_p) = T_{p}( \Phi\circ e )(\hat H_p) = T_{e(p)} \Phi\circ T_p e (\hat H_p). 
\end{align}
Since connections are in direct correspondence to horizontal lifts, it comes with no surprise that inducing connections on $L(M)$ implies inducing horizontal lifts on $L(M)$: indeed one can define $\omega:=Te(\hat\omega):TM\rightarrow TL(M)$ and $\tilde\omega:=T\tilde e(\hat\omega):TM\rightarrow TL(M)$. Suppose then that $\omega^{ab}_\mu$ are given, which is equivalent to giving a connection on $P$: it is easy to see that $\tilde H$ is gauge equivalent to $H$ by comparing the coefficients of the vertical part of the two horizontal lifts when written in the same trivialization. Indeed the first horizontal lift is given by
\begin{align}\label{omega}
T_pe(\hat\omega(v)|_{p})=v^\mu(\partial_\mu|_{e(p)}-\omega^{ab}_\mu(x)\rho_{ab}|_{e(p)}),
\end{align} 
where $e(p)=(x,\epsilon_a)=(x,\overset{\scriptscriptstyle(\alpha)}{e}{}_b)\epsilon^b_a$ and the vertical vector field is $\rho_{ab}=\epsilon^c_d\partial{}_{[a}^d\eta{}^{}_{b]c}\in T_{e(p)}e(P)$. The second connection on $L(M)$ is instead given by 
\begin{align}\nonumber
T_p\tilde e(\hat\omega(v)|_{p})=T_{e(p)}\Phi(T_pe(\omega(v)|_{p}))=v^\mu(\partial_\mu|_{\tilde e(p)}-\tilde\omega^{ab}_\mu(x)\tilde \rho_{ab}|_{\tilde e(p)}),
\end{align} 
where $\tilde \rho_{ab}=\tilde\epsilon^c_d\tilde\partial{}_{[a}^d\eta{}^{}_{b]c}$ and
\begin{align}
\tilde\omega^{ab}_\mu=\phi^a_c(\omega^{cd}_\mu\eta_{de}\bar\phi^e_f+\partial_\mu\bar\phi^c_f)\eta^{fb}.
\end{align}
By right translating these connections into the same point (say $e(p)$) and by comparing the coefficients of the vertical parts, which only depend on the base point $x$, one concludes that these two connections are indeed gauge equivalent with respect to the gauge transformation $\phi$, as claimed.
\begin{remark}\label{remarkcoeff}
For future reference, notice that the horizontal lift obtained by the spin frame $(P,\tilde e)$ can be rewritten in the trivialization induced by that spin frame. Knowing that $\tilde e(p)=(x,\overset{\scriptscriptstyle(\alpha)}{\tilde e}_b)\tilde\epsilon'^b_a$, we get
\begin{align}\nonumber
T_p\tilde e(\hat\omega(v)|_{p})=v^\mu(\partial_\mu|_{\tilde e(p)}-\omega^{ab}_\mu(x)\tilde\rho'_{ab}|_{\tilde e(p)}),
\end{align}
where $\tilde\rho'_{ab}=\tilde\epsilon'^c_d\tilde\partial'{}_{[a}^d\eta{}^{}_{b]c}$ are the vertical vector fields on $T\tilde e(P)$. Notice now that, when written in the two induced trivializations, the coefficients of the vertical part of both connections on $L(M)$ are the same as the ones of $\hat\omega$. 
\end{remark}
This last discussion showed how to induce connections on the frame bundle, once the coefficients $\omega^{ab}_\mu$ are given: they can be, in theory, completely unrelated to the induced metric structure. As a consequence, the Dirac equation built out of it, will depend on the metric only through the spin frame in front of the covariant derivative of the spinor field.\\
We now move on to analyse the case in which the connection depends on the induced metric structure.


\section{Projectable connections on the frame bundle}\label{sec4}
In this Section we shall discuss the case where principal connections on the spin bundle are related to the induced metric by spin frames. This different approach consists on starting with connections on the frame bundle and pulling them back to the spin bundle. Furthermore we will exploit the affine structure of connections to study how to induce different connections on $P$.\\ 
Let us then start from a connection on the frame bundle $H\in TL(M)$: since we are not considering spin frames yet, we start by writing the related horizontal lift in the natural trivialization. Indeed a generic point can be written as $(x,\epsilon_a)=(x,\partial_\mu)\epsilon^\mu_a$ and
\begin{align}
\omega(v)|_{(x,\epsilon_a)}=v^\mu\left(\partial_\mu|_{(x,\epsilon_a)}-\omega^\alpha_{\beta\mu}(x)\rho^\beta_\alpha |_{(x,\epsilon_a)}\right),
\end{align}
where $\rho^\beta_\alpha|_{(x,\epsilon_a)}=\epsilon^\beta_a\partial^a_\alpha$ are the right invariant vector fields on $L(M)$. Consider now the following lemma.
\begin{lemma}
Let H be a connection on $L(M)$ and $(P,e)$ a spin frame. $H$ is tangent to $e(P)=SO_e(M,g)$ in a point $(x,\epsilon_a)$ (and hence in all of its points) if and only if $$\omega^{(ab)}_\mu:=\eta^{c(a}\omega^{b)}{}_{c\mu}=0,$$ where the $(ab)$ indicates that the indices are symmetrised and 
\begin{align}\label{omegae}
\omega^b{}_{c\mu}(x)=e^b_\alpha(x)\left(\omega^\alpha_{\beta\mu}(x)e^\beta_c(x)+\partial_\mu e^\alpha_c(x) \right).
\end{align}
\end{lemma}
\begin{proof}
We begin by expressing our horizontal vector in the trivialization induced by the spin frame $(P,e)$. If $(x,\overset{\scriptscriptstyle(\alpha)}{e}_a)=(x,\partial_\mu)\overset{\scriptscriptstyle(\alpha)}{e}{}^\mu_a$ is such section, we get
\begin{align}\nonumber
\omega(v)|_{(x,\epsilon_a)}=v^\mu(\partial_\mu|_{(x,\epsilon_a)}-\omega^{a}{}_{b\mu}(x)\rho_{a}^b|_{(x,\epsilon_a)}),
\end{align}
where $\omega^{a}{}_{b\mu}(x)$ is exactly given by \eqref{omegae}. Notice that we suppressed the index $\alpha$ labelling the section $\overset{\scriptscriptstyle(\alpha)}{\sigma}$ on $P$ to ease the reading. We already analysed how a change of trivialization on $P$ influences the induced connection in the proof of Proposition \ref{propF(P)}.\\
We are now ready to first prove that $\omega^{(ab)}_\mu=0\implies H\in T_{(x,\epsilon_a)}SO_e(M,g)$: the right invariant vector fields on $SO_e(M,g)$ are 
\begin{align}
\rho_{[ab]}=\eta_{c[a}\rho^c_{b]},
\end{align}
where $[ab]$ indicates skew symmetric indices. In light of this, we rewrite the lifted vectors as
\begin{align}\nonumber
\omega(v)|_{(x,\epsilon_a)}&=v^\mu\left(\partial_\mu|_{(x,\epsilon_a)}-\omega^{ab}_{\mu}(x)\rho_{ab}|_{(x,\epsilon_a)} \right)\\ \nonumber
&=v^\mu\left(\partial_\mu|_{(x,\epsilon_a)}-\omega^{ab}_{\mu}(x)(\rho_{(ab)}+\rho_{[ab]})|_{(x,\epsilon_a)} \right)\\ \nonumber
&=v^\mu\left(\partial_\mu|_{(x,\epsilon_a)}-\omega^{[ab]}_{\mu}(x)\rho_{[ab]}|_{(x,\epsilon_a)} \right),
\end{align}
where in the last step we made use of the hypothesis.\\
On the contrary, take a horizontal vector of $L(M)$
\begin{align}\nonumber
\omega(v)|_{(x,\epsilon_a)}&=v^\mu\left(\partial_\mu|_{(x,\epsilon_a)}-\omega^{ab}_{\mu}(x)\rho_{ab}|_{(x,\epsilon_a)} \right)\\ \nonumber
&=v^\mu\left(\partial_\mu|_{(x,\epsilon_a)}-\omega^{(ab)}_{\mu}(x)\rho_{(ab)}|_{(x,\epsilon_a)}-\omega^{[ab]}_{\mu}(x)\rho_{[ab]}|_{(x,\epsilon_a)} \right).
\end{align}
If we want it to be tangent to $SO_e(M,g)$, we need to require the coefficient of $\rho_{(ab)}$, which are not right invariant vector fields on $SO_e(M,g)$, to vanish. We then get our thesis $\omega^{(ab)}_{\mu}(x)=0$.
\end{proof}
In light of the previous lemma, we now prove the following theorem, which constrains the vertical part of any connection of $L(M)$ to be in a specific form, namely a contorsion-type tensor belonging to $A_2(TM)\otimes T^*M$, if we want it to project on an orthonormal subbundle $e(P)=SO_e(M,g)$.
\begin{theorem}\label{project}
Let $H$ be a connection of $L(M)$, $v\in TM$ and $\{g\}$ be the horizontal lift induced by Levi-Civita connection of the metric $g$. $H$ then projects on $e(P)=SO_e(M,g)$ if and only if 
\begin{align}
\omega(v)|_{(x,\epsilon_a)}=(\{g\}(v)+K(v))|_{(x,\tilde\epsilon_a)},
\end{align}
where $(x,\epsilon_a)\in e(P)$ and $K(v)$ is defined as
\begin{align}\nonumber
K: \ &TM\to V(L(M))\\ \nonumber
&v\mapsto g^{\alpha\gamma}K_{\gamma\beta\mu}v^\mu\rho_\alpha^\beta,
\end{align}
in terms of the contorsion tensor whose coefficients are given by
\begin{align} \nonumber
K_{\gamma\beta\mu}=\frac{1}{2}\Big (g(\partial_\beta,T(\partial_\mu,\partial_\gamma))+g(\partial_\mu,T(\partial_\beta,\partial_\gamma))+g(\partial_\gamma,T(\partial_\beta,\partial_\mu))\Big).
\end{align}
\end{theorem}
\begin{proof}
In order to prove this theorem, we need to show that, under our assumptions, 
\begin{align}
\omega^\alpha_{\beta\mu}=\{g\}^\alpha_{\beta\mu}+g^{\alpha\gamma}K_{\gamma\beta\mu},
\end{align}
where $\{g\}^\alpha_{\beta\mu}$ are the Christoffel symbols of the induced metric $g$. This result follows from the previous lemma: indeed, by analysing the implications of the requirement $\omega^{(ab)}_\mu=0$, one finds
\begin{align}\nonumber
\omega^{ab}_\mu+\omega^{ba}_\mu&=e^a_\alpha\omega^\alpha_{\beta\mu}e^{\beta b}+e^b_\alpha\omega^\alpha_{\beta\mu}e^{\beta a}+e^a_\alpha\partial_\mu e^{\alpha b}+e^b_\alpha\partial_\mu e^{\alpha a}\\ \nonumber
&=e^a_\alpha\omega^\alpha_{\beta\mu}e^{\beta b}+e^b_\alpha\omega^\alpha_{\beta\mu}e^{\beta a}+e^a_\alpha\partial_\mu e^{\alpha b}+e^b_\alpha\partial_\mu(g^{\alpha\lambda}e_\lambda^a)=0.
\end{align}
This implies that 
\begin{align}\nonumber
\partial_\mu g_{\rho\sigma}=g_{\sigma\lambda}\omega^\lambda_{\rho\mu}+g_{\rho\alpha}\omega^{\alpha}_{\sigma\mu}.
\end{align}
By cyclicly permuting the indices and by appropriately summing and subtracting the obtained relations, we get
\begin{align}\nonumber
-\partial_\rho g_{\sigma\mu}+\partial_\mu g_{\rho\sigma}+\partial_\sigma g_{\mu\rho}=g_{\mu\lambda}T^{\lambda}_{\rho\sigma}+g_{\sigma\lambda}T^{\lambda}_{\rho\mu}+g_{\rho\lambda}T^{\lambda}_{\mu\sigma}+2g_{\rho\lambda}\omega^\lambda_{\sigma\mu},
\end{align}
where $T^\lambda_{\alpha\beta}:=\omega^\lambda_{\alpha\beta}-\omega^\lambda_{\beta\alpha}$ is the torsion of the connection. It is skew symmetric in its lower indices by construction. Then one easily gets
\begin{align}\nonumber
\omega^\gamma_{\sigma\mu}&=\{g\}^\gamma_{\sigma\mu}+\frac{1}{2}g^{\gamma\rho}(g_{\mu\lambda}T^{\lambda}_{\sigma\rho}+g_{\sigma\lambda}T^{\lambda}_{\mu\rho}+g_{\rho\lambda}T^{\lambda}_{\sigma\mu})\\ \nonumber
&=\{g\}^\gamma_{\sigma\mu}+\frac{1}{2}g^{\gamma\rho}(T_{\mu\sigma\rho}+T_{\sigma\mu\rho}+T_{\rho\sigma\mu}),
\end{align}
which is our thesis. \\
Before proving the opposite implication, we remark that the contorsion is skew symmetric in its first two indices, meaning 
\begin{align}\label{Kindices}
K_{(\alpha\beta)\gamma}=0,
\end{align}
as a consequence of the property of the torsion. Then
\begin{align}\nonumber
\omega^{ab}_\mu=e^a_\alpha(\{g\}^\alpha_{\beta\mu}e^{\beta b}+\partial_\mu e^{\alpha b})+e^{a\gamma}K_{\gamma\beta\mu}e^{\beta b}\implies \omega^{(ab)}_\mu=0,
\end{align}
where the first term can be easily shown to be skew symmetric by explicitly writing the Christoffel symbols, while the second is a consequence of \eqref{Kindices}.
\end{proof}
The content of this theorem can be summed up by saying that projectable connections form an affine subspace of all connections or that the space of projectable  connection is modelled on {\it contorsion-type
tensors}, i.e.~it is a submodule of dimension $\frac{m^2}{2}(m-1)$ in the module of connections, which has instead dimension $m^3$.\\
As a consequence of this key theorem, we now have the following two corollaries.
\begin{corollary}
A torsionless metric connection $\omega=\{g\}$ trivially projects on $SO_e(M,g)$. 
\end{corollary}
\begin{corollary}
For any vector valued tensor $I: \ TM\times TM\to TM$ satisfying $I(X,Y)=-I(Y,X)$ for any $X,Y\in TM$, one can obtain one and only one projectable connection on $SO_e(M,g)$.
\end{corollary}
\begin{proof}
Consider indeed the combination
\begin{align} \nonumber
\omega^\alpha_{\beta\mu}=\{g\}^\alpha_{\beta\mu}+\frac{1}{2}g^{\alpha\gamma}(I_{\beta\mu\gamma}+I_{\mu\beta\gamma}+I_{\gamma\beta\mu}).
\end{align}
By taking the skew symmetric part of this equation, we get
\begin{align}
T^\alpha_{\beta\mu}=I^\alpha_{\beta_\mu},
\end{align}
so we see that the chosen tensor $I^\alpha_{\beta\mu}$ assumes in fact the role of the torsion of the connection. In light of the previous theorem, we can conclude our thesis.
\end{proof}
This last corollary allows us to build a projectable connection with torsion from a torsionless one, as long as one has a skew symmetric tensor to begin with.	\\
Once a connection is projected on the subbundle $e(P)$, it is straightforward to pull it back to the spin bundle through the spin frame $(P,e)$: this procedure defines a connection $\hat H$ whose horizontal vectors are given by
\begin{align}\label{pullbackconn}
\hat\omega(v)|_{p}=v^\mu(\partial_\mu|_{p}-\omega^{ab}_\mu(x)\sigma_{ab}|_p),
\end{align}
where the coefficients $\omega^{ab}_\mu(x)$ are 
\begin{align}
\omega^{ab}_{\mu}(x)=e^a_\alpha(x)\left((\{g\}^\alpha_{\beta\mu}+g^{\alpha\gamma}K_{\gamma\beta\mu})e^\beta_c(x)+\partial_\mu e^\alpha_c(x) \right)\eta^{cb}.
\end{align}
as a consequence of Theorem \ref{project}.\\
Suppose now you have another spin frame $(P,\tilde e)$: since everything we proved until now is valid for any spin frame, one would repeat the same arguments, which we now briefly sum up. Take a connection $\tilde H$ on the frame bundle, project it on $\tilde e(P)$ and pull it back on the spin bundle. This procedure clearly yields
\begin{align}
\begin{cases}
\hat{\tilde\omega}(v)|_{p}=v^\mu(\partial_\mu|_{p}-\Omega^{ab}_\mu(x)\sigma_{ab}|_p),\\
\Omega^{ab}_{\mu}(x)=\tilde e^a_\alpha(x)\left((\{ \tilde g\}^\alpha_{\beta\mu}+\tilde g^{\alpha\gamma}\tilde K_{\gamma\beta\mu})\tilde e^\beta_c(x)+\partial_\mu \tilde e^\alpha_c(x) \right)\eta^{cb},
\end{cases}
\end{align}
where $\tilde g$ is the new induced metric and $\tilde K_{\gamma\beta\mu}$ is the contorsion of the new connection, completely unconstrained. If no spin frame transformation is given, the obtained connection will be different from \eqref{pullbackconn}. We will come back on this later on, when we will discuss Dirac equations.\\
As we will now show, spin frame transformations have the property to select a particular connection of all the possible $\tilde H$, with the same pullback \eqref{pullbackconn} on $P$. Take indeed $T\Phi(H)$: this induced connection is clearly different from the original one, as discussed previously and we can see if it is tangent to $\tilde e(P)$. In order to do that, we first prove the following lemma.
\begin{lemma}\label{lemmah}
Let $\{g\}$ and $\{\tilde g\}$ be two horizontal lifts induced by the Levi-Civita connections of the metric structures defined by $(P,e)$ and $(P,\tilde e)$. Then
\begin{align}\nonumber
(\{\tilde g\}-\{g\})(v)|_{(x,\tilde\epsilon_a)}=h(v)|_{(x,\tilde\epsilon_a)},
\end{align}
where $h: \ TM\to V(L(M))$ satisfies
\begin{align}
h(v)=v^\mu\bigg(\phi^\alpha_\gamma\overset{\{g\}}{\nabla}^{\ }_{(\beta}\bar\phi^\gamma_{\mu)}+\phi^\alpha_\gamma g^{\gamma\delta}\phi^{\lambda}_\delta g_{\rho\sigma}\bar\phi^\rho_{(\beta}\overset{\{g\}}{\nabla}^{}_{\mu)}\bar\phi^\sigma_\lambda-\phi^{\alpha}_\gamma g^{\gamma\delta}\phi^\lambda_\delta g_{\rho\sigma}\overset{\{g\}}{\nabla}^{}_\lambda\bar\phi^\rho_{(\beta}\bar\phi^\sigma_{\mu)}\bigg)\tilde\rho^\beta_\alpha.
\end{align}
Notice that we suppressed the dependence of $\phi$ in order to ease the reading.
\end{lemma}
\begin{proof}
We use the fact that both Levi-Civita connections are compatible with the corresponding metrics, i.e. in the natural trivialization $$\overset{\{g\}}{\nabla}_\mu g_{\alpha\beta}=0, \qquad\overset{\{\tilde g\}}{\nabla}_\mu \tilde g_{\alpha\beta}=0,$$ where the two covariant derivatives are the ones induced on the associated bundle $Lor(M)$. Take now a tensor in $TM \otimes S_2(M)$, whose coefficients will be called $h^\alpha_{\beta\mu}$ and satisfy $h^\alpha_{\beta\mu}= \{\tilde g\}^\alpha_{\beta\mu} -\{g\}^\alpha_{\beta\mu} $, then
\begin{align}\nonumber
0=&\overset{\{\tilde g\}}{\nabla}_\mu  \tilde g_{\alpha\beta} = \overset{\{g\}}{\nabla}_\mu \tilde g_{\alpha\beta}-h^\epsilon_{\alpha\mu}  \tilde g_{\epsilon\beta}  - h^\epsilon_{\beta\mu}  \tilde g_{\alpha\epsilon}\\ \nonumber
&=\overset{\{g\}}{\nabla}_\mu  \bar \phi_\alpha^\rho  g_{\rho\sigma}   \bar \phi_\beta^\sigma +   \bar \phi_\alpha^\rho  g_{\rho\sigma} \overset{\{g\}}{\nabla}_\mu  \bar  \phi_\beta^\sigma  -  \tilde g_{\beta\epsilon} h^\epsilon_{\alpha\mu} -\tilde g_{\alpha\epsilon}  h^\epsilon_{\beta\mu} \\ \nonumber
&=  \tilde g_{\beta\sigma}   \phi_\rho^\sigma\overset{\{g\}}{\nabla}_\mu  \bar \phi_\alpha^\rho + \tilde g_{\alpha\rho} \phi_\sigma^\rho \overset{\{g\}}{\nabla}_\mu\bar \phi_\beta^\sigma - \tilde h_{\beta\alpha\mu}   - \tilde h_{\alpha\beta\mu},
\end{align}
where we set $\tilde h_{\beta \alpha\mu}:= \tilde g_{\beta\epsilon} h^\epsilon_{\alpha\mu}$. This is recast as
$$
\tilde h_{\beta\alpha\mu}+\tilde  h_{\alpha\beta\mu}=\tilde g_{\beta\sigma}\phi_\rho^\sigma\overset{\{g\}}{\nabla}_\mu\bar \phi_\alpha^\rho + \tilde g_{\alpha\rho} \phi_\sigma^\rho\overset{\{g\}}{\nabla}_\mu\bar\phi_\beta^\sigma  
 =2\phi_\rho^\sigma\overset{\{g\}}{\nabla}_\mu\bar \phi_{(\alpha}^\rho\tilde g^{\phantom{\rho}}_{\beta)\sigma}.
 $$
Since $\tilde h_{\alpha[\beta\mu]}=0$, by cyclically permuting the indices, we obtain 
\begin{align}\nonumber
 \tilde h_{\alpha\beta\mu }=\tilde g_{\alpha\sigma}\phi_\rho^\sigma\overset{\{g\}}{\nabla}_{(\mu}  \bar \phi_{\beta)}^\rho + \phi_\sigma^\rho  \tilde g_{\rho(\beta} \overset{\{g\}}{\nabla}_{\mu)}\bar\phi_\alpha^\sigma-\phi_\rho^\sigma\overset{\{g\}}{\nabla}_\alpha\bar\phi_{(\mu}^\rho\tilde g^{\phantom{\rho}}_{\beta)\sigma},
\end{align}
with the expected symmetry in the last two indices. As a result we get 
\begin{align}
\{\tilde g\}^\alpha_{\beta\mu}-\{g\}^\alpha_{\beta\mu}=\phi^\alpha_\gamma\overset{\{g\}}{\nabla}_{(\beta}\bar\phi^\gamma_{\mu)}+\phi^\alpha_\gamma g^{\gamma\delta}\phi^{\lambda}_\delta g_{\rho\sigma}\bar\phi^\rho_{(\beta}\overset{\{g\}}{\nabla}_{\mu)}\bar\phi^\sigma_\lambda-\phi^{\alpha}_\gamma g^{\gamma\delta}\phi^\lambda_\delta g_{\rho\sigma}\overset{\{g\}}{\nabla}_\lambda\bar\phi^\rho_{(\beta}\bar\phi^\sigma_{\mu)},
\end{align}
which concludes our proof: indeed, this is exactly the expression that shows up in the difference $(\{\tilde g\}-\{g\})(v)|_{(x,\tilde\epsilon_a)}$, leading to our thesis.
\end{proof}
The second step is to consider the difference between two vectors lifted through $H$ and $T\Phi (H)$. 
\begin{lemma}\label{lemmak}
Let $\omega$ and $\tilde \omega$ two horizontal lift induced by $H$ and $T\Phi(H)$ respectively. Then
\begin{align}\nonumber
(\tilde\omega-\omega)(v)|_{(x,\tilde\epsilon_a)}=k(v)|_{(x,\tilde\epsilon_a)},
\end{align}
where $k:TM\rightarrow V(L(M))$ satisfies
\begin{align}
k(v)=v^\mu\bigg(\phi^\alpha_\gamma\overset{\omega}{\nabla}_\mu\bar\phi^\gamma_\beta\bigg )\tilde\rho^\beta_\alpha,
\end{align}
$\overset{\omega}{\nabla}$ being the covariant derivative induced by the connection $H$ on the associated bundle of spin frame transformations.\\
We stress out that we are not requiring the connections to be projected on any orthonormal subbundle.
\end{lemma}
\begin{proof}
As a starter, we compute the coefficients of the vertical part of $\tilde\omega(v)$: we have
\begin{align}\label{projectedcoeff}
\tilde\omega^\alpha_{\beta\mu}(x)=\phi^\alpha_\gamma(x)\left(\omega^\gamma_{\delta\mu}(x)\bar\phi_\beta^\delta(x)+\partial_\mu\bar\phi^\gamma_\beta(x)\right).
\end{align}
Take now a tensor in $T(M)\otimes T^*(M)\otimes T^*(M)$ whose coefficients will be called $k^\alpha_{\beta\mu}$ and satisfy $$k^\alpha_{\beta\mu}=\tilde\omega^\alpha_{\beta\mu}-\omega^\alpha_{\beta\mu}.$$ By rearranging the terms one easily gets
\begin{align}
k^\alpha_{\beta\mu}=\phi^\alpha_\gamma\overset{\omega}{\nabla}_\mu\bar\phi^\gamma_\beta,
\end{align}
which is what one encounters when computing $(\tilde\omega-\omega)(v)|_{(x,\tilde\epsilon_a)}$.\\
As a side note, we give a second proof of the same result: for this proof we additionally require $H$ and $T\Phi(H)$ to be the image of a connection $\hat H$ on the spin bundle through two spin frames $(P,e)$ and $(P,\tilde e)$. This is also equivalent to the conditions $\overset{\hat\omega,\omega}{\nabla}_\mu  e^\nu_a=\overset{\hat\omega,\tilde\omega}{\nabla}_\mu \tilde e^\nu_a=0$. From the second condition we have
\begin{align}\nonumber
0=\overset{\hat\omega,\tilde\omega}{\nabla}_\mu \tilde e^\nu_a=\overset{\hat\omega,\omega}{\nabla}_\mu \tilde e^\nu_a+k^\nu_{\alpha\mu}\tilde e^\alpha_a=\left(\overset{\omega}{\nabla}_\mu\phi^\nu_\alpha+k^\nu_{\beta\mu}\phi^\beta_\alpha\right)e^\alpha_a\implies k^\alpha_{\beta\mu}=\phi^\alpha_\gamma\overset{\omega}{\nabla}_\mu\bar\phi^\gamma_\beta,
\end{align}
which is indeed our result.
\end{proof}
We are now ready to write down all possible relations among the relevant connections. In particular, we want to consider projected connections on two orthonormal subbundles $e(P)$ and $\tilde e(P)$. By doing so, we obtain the following theorem, which states what we claimed previously, that spin frame transformations select a particular set of contorsions which allow to pull the connections back onto a single connection in $P$.
\begin{theorem}\label{omegaprimeproject}
Let $H$ be a connection on the frame bundle, projectable on the subbundle $e(P)$ and take $T\Phi(H)$. The latter connection is projectable on $\tilde e(P)$ if and only if its contorsion satisfies
\begin{align}\label{contorsionprime}
\tilde K(v)|_{(x,\tilde\epsilon_a)}=(K(v)+k(v)-h(v))|_{(x,\tilde\epsilon_a)}.
\end{align}
\end{theorem}
\begin{proof}
The proof of this statement is now straightforward, having proved the previous lemmas. Indeed
\begin{align}\nonumber
\tilde\omega(v)|_{(x,\tilde\epsilon_a)}&=(\omega(v)+k(v))|_{(x,\tilde\epsilon_a)}=(\{g\}(v)+K(v)+k(v))|_{(x,\tilde\epsilon_a)}\\ \nonumber
&=(\{\tilde g\}(v)+\tilde K(v))|_{(x,\tilde\epsilon_a)},
\end{align}
where in the second line we used the fact that $T\Phi(H)$ is projectable. From this we get  
\begin{align}\nonumber
\tilde K(v)|_{(x,\tilde\epsilon_a)}=(K(v)+k(v)-h(v))|_{(x,\tilde\epsilon_a)}.
\end{align}
Vice versa, recast \eqref{contorsionprime} in the following way
\begin{align}\nonumber
\tilde K(v)|_{(x,\tilde\epsilon_a)}&=(K(v)+k(v)-\{\tilde g\}(v)+\{ g\}(v))|_{(x,\tilde\epsilon_a)}\\ \nonumber
\implies (\{\tilde g\}(v)+\tilde K(v))|_{(x,\tilde\epsilon_a)}&=(\{ g\}(v)+K(v)+k(v))|_{(x,\tilde\epsilon_a)}=(\omega(v)+k(v))|_{(x,\tilde\epsilon_a)}=\tilde\omega(v)|_{(x,\tilde\epsilon_a)}.
 \end{align}
 In the first line we used lemma \ref{lemmah}, whereas in the second line we used the fact that $H$ is projectable and lemma \ref{lemmak}. We then see that $T\Phi(H)$ is projectable on $\tilde e(P)$.
\end{proof}
\begin{remark}
We can give an explicit expression for the vertical map $\tilde K$: indeed, in the natural trivialization we have
\begin{align}\nonumber
\tilde K(v)|_{(x,\tilde\epsilon_a)}&=v^\mu\bigg(g^{\alpha\gamma}K_{\gamma\beta\mu}+\phi^\alpha_\gamma\overset{\omega}{\nabla}_\mu\bar\phi^\gamma_\beta-\phi^\alpha_\gamma\overset{\{g\}}{\nabla}^{\ }_{(\beta}\bar\phi^\gamma_{\mu)}\\ \nonumber
&-\phi^\alpha_\gamma g^{\gamma\delta}\phi^{\lambda}_\delta g_{\rho\sigma}\bar\phi^\rho_{(\beta}\overset{\{g\}}{\nabla}^{}_{\mu)}\bar\phi^\sigma_\lambda+\phi^{\alpha}_\gamma g^{\gamma\delta}\phi^\lambda_\delta g_{\rho\sigma}\overset{\{g\}}{\nabla}^{}_\lambda\bar\phi^\rho_{(\beta}\bar\phi^\sigma_{\mu)}\bigg)\tilde\rho_\alpha^\beta|_{(x,\tilde\epsilon_a)} \\ \nonumber
&=v^\mu\bigg(\phi^\alpha_\gamma\overset{\{g\}}{\nabla}^{\ }_{[\mu}\bar\phi^\gamma_{\beta]}-\phi^\alpha_\gamma g^{\gamma\delta}\phi^{\lambda}_\delta g_{\rho\sigma}\bar\phi^\rho_{(\beta}\overset{\{g\}}{\nabla}^{}_{\mu)}\bar\phi^\sigma_\lambda\\ \nonumber
&+\phi^{\alpha}_\gamma g^{\gamma\delta}\phi^\lambda_\delta g_{\rho\sigma}\overset{\{g\}}{\nabla}^{}_\lambda\bar\phi^\rho_{(\beta}\bar\phi^\sigma_{\mu)}+\phi^\alpha_\gamma g^{\gamma\rho}K_{\rho\lambda\mu}\phi^\lambda_\beta \bigg)\tilde\rho_\alpha^\beta|_{(x,\tilde\epsilon_a)}.
\end{align}
This relation is equivalent to giving an expression for the contorsion
\begin{align}\label{tildeK}
\tilde K_{\rho\beta\mu}=g_{\alpha\gamma}\bar\phi^\alpha_\rho\overset{\{g\}}{\nabla}_{[\mu}\bar\phi^\gamma_{\beta]}-g_{\alpha\gamma}\bar\phi^\alpha_{(\beta}\overset{\{g\}}{\nabla}_{\mu)}\bar\phi^\gamma_\rho+g_{\alpha\gamma}\overset{\{g\}}{\nabla}_\rho\bar\phi^\alpha_{(\beta}\bar\phi^\gamma_{\mu)}+\bar\phi^\sigma_\rho K_{\sigma\eta\mu}\bar\phi^\eta_\beta.
\end{align}
The obtained expression is indeed a contorsion as it has the expected tensor properties: it is skew symmetric in the first two indices $$\tilde K_{(\rho\beta)\mu}=0,$$ as one can see if we rewrite it in a way that is more suitable for that scope $$K_{\rho\beta\mu}=\phi_\lambda^\sigma\overset{\{g\}}{\nabla}_{\mu}  \bar \phi_{[\beta}^\lambda  \tilde g^{}_{\rho]\sigma}  
- \phi_\lambda^\sigma \tilde g_{\sigma[\rho} \overset{\{g\}}{\nabla}_{\beta]}  \bar \phi_{\mu}^\lambda    
+\tilde g_{\mu\sigma}  \phi_\lambda^\sigma\overset{\{g\}}{\nabla}^{}_{[\rho}  \bar \phi_{\beta]}^\lambda
+  \bar\phi^\sigma_{[\rho}  \bar \phi^\lambda_{\beta]} K_{\sigma\lambda\mu}. $$
\end{remark}
Among all possible projectable connections along $\tilde g$ (i.e.~for any possible contorsion tensor $\tilde K$),
equation \eqref{tildeK} determines the unique contorsion (for any spin frame transformation) to define a connection $\tilde H$ which is pulled back onto the same connection $\hat \omega$ on $P$.\\
In other words, the set of connections associated to $\hat \omega$ is left invariant by spin frame transformations.
Theorem \ref{omegaprimeproject} can be restated in terms of the torsion of the transported connection, being the relation with the contorsion invertible.
\begin{corollary}
The transported connection $T\Phi(H)$ is projectable on $\tilde e(P)$ if and only if 
\begin{align}\label{torsionprime}
\tilde T^\lambda_{\beta\mu}=2\phi^\lambda_\gamma\overset{\{g\}}{\nabla}_{[\mu}\bar\phi^\gamma_{\beta]}+2\phi^\lambda_\gamma g^{\gamma\sigma}K_{\sigma\eta[\mu}\bar\phi^\eta_{\beta]}.
\end{align}
\end{corollary}
\begin{proof}
The thesis is obtained by just taking the antisymmetric part of \eqref{tildeK}
\begin{align}\nonumber
\tilde K_{\rho[\beta\mu]}=\frac{1}{2}\tilde T_{\rho\beta\mu}=g_{\alpha\gamma}\bar\phi^\alpha_\rho\overset{\{g\}}{\nabla}_{[\mu}\bar\phi^\gamma_{\beta]}+\bar\phi^\sigma_\rho K_{\sigma\eta[\mu}\bar\phi^\eta_{\beta]}.
\end{align}
From this relation one can easily get an expression for the torsion, by multiplying for $\tilde g^{\rho\lambda}$
\begin{align}\nonumber
\tilde T^\lambda_{\beta\mu}=2\phi^\lambda_\gamma\overset{\{g\}}{\nabla}_{[\mu}\bar\phi^\gamma_{\beta]}+2\phi^\lambda_\gamma g^{\gamma\sigma}K_{\sigma\eta[\mu}\bar\phi^\eta_{\beta]}.
\end{align}
The contrary is easily proved by inserting \eqref{torsionprime} into the definition of the contorsion.
\end{proof}
As a particular case, one could decide to start from a projectable, torsionless connection $H$, thus a Levi-Civita connection for $e(P)$. In this case, it is clear that the transported connection $T\Phi(H)$ cannot be torsionless as the starting one, if we want it to be projectable on $\tilde e(P)$. Indeed its torsion must in general follow a particular case of equation \eqref{torsionprime} where the contorsion is zero. This means that one obtains the very compact result
\begin{align}
T'^\lambda_{\beta\mu}=\phi^\lambda_\gamma\left(\overset{\{g\}}{\nabla}_{\mu}\bar\phi^\gamma_{\beta}-\overset{\{g\}}{\nabla}_{\beta}\bar\phi^\gamma_{\mu}\right).
\end{align}
\begin{remark}\label{finalremark}
We saw how Theorem \ref{contorsionprime} explicitly selects, out of all possible $\tilde H$s, a projectable connection with a particular property: the pullback on $P$ through $(P,\tilde e)$ of such connection is the same as the one obtained from $H$ along $(P,e)$, as we briefly mentioned previously. Let us now show this by noticing that 
\begin{align}\nonumber
\omega^{ab}_\mu=e^b_\alpha\left([\{g\}^\alpha_{\beta\mu}+g^{\alpha\gamma}K_{\gamma\beta\mu}]e^\beta_c+\partial_\mu e^\alpha_c \right)=\tilde e^b_\alpha\left([\{\tilde g\}^\alpha_{\beta\mu}+\tilde g^{\alpha\gamma}\tilde K_{\gamma\beta\mu}(\tilde e,\phi, K)]\tilde e^\beta_c+\partial_\mu \tilde e^\alpha_c \right).
\end{align}
This equality can be proved by using \eqref{tildeK} and the explicit dependence of $\tilde g$ and $\tilde e$ in terms of $\Phi$ and it agrees with Remark \ref{remarkcoeff}, where we stated that the coefficients of the two connections on the frame bundle are the same, when written in the trivialization induced by the two spin frames.\\
As a consequence, the lifted horizontal vector on $P$ will be in both cases
\begin{align}
\hat\omega(v)|_p=v^\mu(\partial_\mu|_p-\omega^{ab}_\mu\sigma_{ab}|_p).
\end{align}
However, we can clearly see the difference between this second frame bundle approach and the one discussed in the previous Section: in that case the coefficients of $\hat\omega(v)$ were known from the beginning and they were in principle unrelated to any metric structure, whereas in this case we obtain an explicit expression in terms of $e$ and a contorsion $K$, or in terms of $\tilde e$ and $\phi$, depending on which expression is preferred. We will elaborate more on this last comment in the final Section.
\end{remark}
Having discussed similarities and differences between the two approaches, the one starting from the spin bundle and the one starting from the frame bundle, we now move on drawing up some conclusions.


\section{Conclusions and perspectives}
In this paper we reviewed spin frames, which are defined on a manifold and a signature, namely  $(M; r, s)$. In this framework, the metric is not a fundamental structure anymore as it is instead induced from spin frames. We indeed have a bundle map $g: F(P) \rightarrow Lor(M; r, s),$ which is locally given by \eqref{localg}. That is a global vertical bundle map, so it also maps sections (i.e.~spin frames) onto sections (i.e.~metrics of the selected signature). \\
Spin frames are more general than standard spin structures, they in fact allow different metrics, but not as general as the double covering formalism on $\tilde L(M)$. For example, take the two metrics 
\begin{align}\nonumber
g_1=& -dr^2 + r^2 d\Omega^2\\ \nonumber
g_2=& \cos(2\theta) (-dr^2 + r^2 d\theta^2) + r^2 \sin^2(\theta) d\varphi^2  + 2r\sin(2\theta) dr\cdot d\theta, 
\end{align}
both defined on $M={\mathbb R}^3-\{0\}$ and both written in spherical coordinates (the latter being the flat Minkowski metric $g_2 = dx^2+dy^2 - dz^2$ written in such coordinates). These two metrics cannot be obtained as spin frames on the same gauge bundle. As a matter of fact $g_2$ is obtained as a spin frame on the trivial bundle, while the gauge bundle for $g_1$ is not trivial (if it were trivial, there would exist a global frame that would induce a non-zero vector field on a sphere $S^2$). \\
The whole spin frame framework is clearly covariant with respect to spin transformations $Aut(P)$, as spin structures were, but allows the action of a different kind of transformations, namely spin frame transformations $Aut_V(L(M))$.\\
In particular we studied two approaches: in the first one we saw how spin frames are able to induce connections on the frame bundle as $\omega = Te(\hat \omega)$. In this view, the geometry on $M$ is described as in Palatini theories, in which the connection (as well as parallel transport and autoparallel trajectories) is a priori independent of the metric structure. \\
In the second approach, we described how, starting from the frame bundle, one can project and induce connections on $P$, which instead depend on the metric structure. In particular, we saw that spin frame transformations give a restriction on the type of torsions (or contorsions) allowed for another connection to project onto the induced othornomal subbudle $T(\Phi\circ e)(P)$. \\
Furthermore, notice that, in this approach, a shift by a contorsion type tensor of a connection projected onto an orthonormal subbundle always preserves this property. In particular, if $H$ and $T\Phi(H)$ are shifted by the same contorsion, they both induce the same connection $\hat H'$ on the spin bundle, for any tensor we add. We argued that this is a submodule in the module of connections.\\
At last, since spin frame transformations do not preserve the induced metric, they can relate Dirac equations for a given spin frame $(P,e)$ to Dirac equations for an inequivalent one $(P,\tilde e)$. Indeed, if we have a triple $(e,\hat H,\hat\lambda)$ and a spin frame transformation $\Phi$ the first Dirac equation
\begin{align}\nonumber
i e_a^\mu \gamma^a \overset{\hat\omega}{\nabla}_\mu \psi + \mu \psi=i e_a^\mu \gamma^a  \overset{\{g\}}{\nabla}_\mu \psi + \mu \psi+  \frac{i}{4}e_a^\mu  K^{bc}_\mu  \gamma^a [\gamma_b, \gamma_c]\psi=0
\end{align}
is mapped into $(\tilde e,\hat H,\hat\lambda)$
\begin{align}
i \tilde e_a^\mu \gamma^a \overset{\hat\omega}{\nabla}_\mu \psi + \mu \psi=i \tilde e_a^\mu \gamma^a  \overset{\{\tilde g\}}{\nabla}_\mu \psi + \mu \psi+  \frac{i}{4}e_a^\mu  \tilde K^{bc}{}_\mu  \gamma^a [\gamma_b, \gamma_c]\psi=0.
\end{align}
Notice that in the last step we chose one of the two possible expressions of the connection, as we were free to do as a consequence of Remark \ref{finalremark}. By using this expression, we see that we are both changing the metric structure appearing in the equations as well as the type of interaction (or we are adding one if we start from a Dirac equation without a contorsion contribution). Besides the geometrical result, we see that changing the metric structure means, from a physical point of view, changing the speed of light, which is described by null vectors with respect to a certain metric. In some contexts, where gravity is used to model graphene from a high energy point of view (see \cite{newpaper} and \cite{Andrianopoli2018}), one would like to introduce, besides the speed of light which naturally appears in the model, another characteristic speed, which would be then interpreted as the speed of quasiparticles in such material. In these situations spin frame transformations would provide a tool to introduce it.\\
Furthermore, the presence of a torsion (and then a contorsion term appearing in the Dirac equations) is related to defects, called dislocations, in the honeycomb lattice structure of graphene. See \cite{Katanaev:1992kh}, \cite{Kleinert} for a deeper analysis. \\
Our purpose will be then to investigate explicit examples, where to apply this construction and eventually to apply them to graphene, which does provide a lab to test Dirac equations on curved geometry, where all global aspects become relevant.

 \section*{Acknowledgement}

LF acknowledges the INFN grant QGSKY, the local research project {\it  Metodi Geometrici in Fisica Matematica e Applicazioni (2017)} of Dipartimento di Matematica of University of Torino (Italy), and the grant INdAM-GNFM. This article is based upon work from COST Action (CA15117 CANTATA), supported by COST (European Cooperation in Science and Technology).\\
The authors finally thank Anna Fino for useful discussions about the different formulations of the Dirac equations.


\appendix
\section{}
\begin{proposition} 
The Dirac equation is covariant with respect to automorphisms of $P$. 
\end{proposition}
\begin{proof}
In a local trivialization such automorphisms act as
\begin{align}
\Phi:P\rightarrow P: [x,S] \mapsto [x'= \phi(x), \varphi(x)\cdot S]
\end{align}
where $\phi:M\rightarrow M$ is the diffeomorphism on which the automorphism $\Phi$ is projecting, and $\varphi: U\rightarrow Spin_e(r, s)$ is a local function defined on the trivialization domain $U\subset M$.\\
The automorphism $\Phi$ acts on spinor fields as $\psi'(x')= \varphi(x)\cdot \psi(x)$ (the representation $\lambda$ being understood), on the connection $\hat H$ and hence on  the covariant derivative
$\hat\nabla' \psi' = \varphi \hat\nabla \psi $. Finally, it also acts on the frame as $e'_a= e_b l^b_a(\varphi^{-1})$, where $l^b_a$ is the matrix representation of the covering map.
Therefore, the Dirac equation for $\psi'$ is
\begin{align}\nonumber
i  e'^\mu_a \gamma^a  \hat\nabla'_\mu \psi' + \mu \psi' &=  i e_b^\mu l^b_a(\varphi^{-1}) \gamma^a \varphi   \hat\nabla_\mu \psi + \mu \varphi \psi =  i  e_b^\mu l^b_a(\varphi^{-1}) \varphi\varphi^{-1}\gamma^a \varphi  \hat \nabla_\mu \psi + \mu \varphi \psi \\ \nonumber
 &=  \varphi\Big( i e_b^\mu l^b_a(\varphi^{-1})  l^a_c(\varphi) \gamma^c  \hat\nabla_\mu \psi + \mu  \psi\Big)= \varphi\Big( i e_a^\mu \gamma^a    \hat\nabla_\mu \psi + \mu  \psi\Big),
\end{align}
where we used the definition of the covering map, namely $\varphi\gamma^a \varphi^{-1}=  l_c^a(\varphi^{-1}) \gamma^c$.\\
We then see that if $\psi$ is a solution of Dirac equation, $\psi'$ is a solution as well.
\end{proof}
\begin{proposition}\label{spintransformationsandmetricstructure}
Any spin transformation, vertical automorphism of the spin bundle $A:P\to P$ does not change the induced metric structure.
\end{proposition}
\begin{proof}
We start the proof by noticing that one can define the metric by just using induced sections: indeed, by writing $p=\overset{\scriptscriptstyle(\alpha)}{\sigma}\cdot\overset{\scriptscriptstyle(\alpha)}{g}$, one sees that $$e(p)=(x,\epsilon_a)=(x,\cdot\overset{\scriptscriptstyle(\alpha)}{e}_b)\hat l^b_a(\overset{\scriptscriptstyle(\alpha)}{g}) $$ and then the definition of the metric can be restated as
\begin{align}\nonumber
g=\overset{\scriptscriptstyle(\alpha)}{e}{}^a_\mu\eta_{ab}\overset{\scriptscriptstyle(\alpha)}{e}{}^b_\nu dx^\mu\otimes dx^\nu.
\end{align}
This last expression is actually independent of the chosen trivialization as it can be easily shown by taking a different section of the spin bundle $\overset{\scriptscriptstyle(\beta)}{\sigma}$ and by using \eqref{changetrivialization}.\\
We now see that any vertical automorphism acts on the point $p\in P$ as
\begin{align}\nonumber
A(p)=A(\overset{\scriptscriptstyle(\alpha)}{\sigma}\cdot\overset{\scriptscriptstyle(\alpha)}{g})=\overset{\scriptscriptstyle(\alpha)}{\sigma}\cdot A\cdot\overset{\scriptscriptstyle(\alpha)}{g}
\end{align}
and its image through the spin frame is given by 
\begin{align}\nonumber
(x,\epsilon'_a)=e(A(p))=e(\overset{\scriptscriptstyle(\alpha)}{\sigma}\cdot A\cdot\overset{\scriptscriptstyle(\alpha)}{g})=e(\overset{\scriptscriptstyle(\alpha)}{\sigma})\cdot\hat l(A\cdot\overset{\scriptscriptstyle(\alpha)}{g})=(x,\overset{\scriptscriptstyle(\alpha)}{e}_b)\hat l^b_a(A\cdot\overset{\scriptscriptstyle(\alpha)}{g}).
\end{align}
At last, we obtain our thesis by computing the induced metric 
\begin{align}\nonumber
g'&=\epsilon'^a_\mu\eta_{ab}\epsilon'^b_\nu dx^\mu\otimes dx^\nu=\overset{\scriptscriptstyle(\alpha)}{e}{}_\mu^c\hat l_c^a(\overset{\scriptscriptstyle(\alpha)}{\bar g}\cdot \bar A)\eta_{ab}\hat l_d^b(\overset{\scriptscriptstyle(\alpha)}{\bar g}\cdot \bar A)\overset{\scriptscriptstyle(\alpha)}{e}{}_\nu^d dx^\mu\otimes dx^\nu \\ \nonumber
&=\overset{\scriptscriptstyle(\alpha)}{e}{}_\mu^c\eta_{cd}\overset{\scriptscriptstyle(\alpha)}{e}{}_\nu^d dx^\mu\otimes dx^\nu=g,
\end{align}
where we indicated with bars the inverses of the group elements of $Spin(\eta)$ and where in the third step we used that $\hat l\ ^b_a(\overset{\scriptscriptstyle(\alpha)}{\bar g}\cdot \bar A)\in SO_e(r,s)\hookrightarrow GL(m,\mathbb R)$
\end{proof}


\end{document}